\documentclass{amsart}
\usepackage{amssymb}

\usepackage[usenames]{color}

\input xy
\xyoption{all}

\newcommand{\w}{\omega}
\newcommand{\e}{\varepsilon}
\newcommand{\IN}{\mathbb N}
\newcommand{\A}{\mathcal A}
\newcommand{\IR}{\mathbb R}
\newcommand{\IQ}{\mathbb Q}
\newcommand{\F}{\mathcal F}
\newcommand{\Ra}{\Rightarrow}
\newcommand{\K}{\mathcal K}
\newcommand{\U}{\mathcal U}
\newcommand{\I}{\mathcal I}
\newcommand{\N}{\mathcal N}
\newcommand{\nw}{\mathrm{nw}}
\newcommand{\Pyt}{\mathfrak P}

\newcommand{\PP}{\mathcal P}
\newcommand{\cs}{\mathrm{cs}}

\newcommand{\cbox}{\boxdot}

\newcommand{\pr}{\mathrm{pr}}
\newcommand{\id}{\mathrm{id}}
\newcommand{\cl}{\mathrm{cl}}

\newtheorem{theorem}{Theorem}[section]
\newtheorem{proposition}[theorem]{Proposition}
\newtheorem{claim}[theorem]{Claim}

\newtheorem{example}[theorem]{Example}
\newtheorem{corollary}[theorem]{Corollary}
\newtheorem{lemma}[theorem]{Lemma}
\theoremstyle{definition}
\newtheorem{definition}[theorem]{Definition}
\newtheorem{remark}[theorem]{Remark}

\title{The strong Pytkeev property in topological spaces}
\author{Taras Banakh, Arkady Leiderman}
\keywords{The strong Pytkeev property, function space with compact-open topology, topological group, topological loop, topological lop, rectifiable space}
\subjclass{54E20; 54C35; 22A30}
\address{T.Banakh: Department of Mathematics, Ivan Franko National University of Lviv (Ukraine) and\newline Instytut Matematyki, Jan Kochanowski University in Kielce (Poland)}
\email{t.o.banakh@gmail.com}
\address{A.Leiderman: Department of Mathematics, Ben-Gurion University of the Negev, Beer Sheva, Israel}
\email{arkady@math.bgu.ac.il}

\begin{document}
\begin{abstract}
A topological space $X$ has the strong Pytkeev property at a point $x\in X$ if there exists a countable family $\mathcal N$ of subsets of $X$ such that for each neighborhood $O_x\subset X$ and subset $A\subset X$ accumulating at $x$, there is a set $N\in\mathcal N$ such that $N\subset O_x$ and $N\cap A$ is infinite. We prove that for any $\aleph_0$-space $X$ and any space $Y$ with the strong Pytkeev property at a point $y\in Y$ the function space $C_k(X,Y)$ has the strong Pytkeev property at the
constant function $X\to \{y\}\subset Y$. If the space $Y$ is rectifiable, then the function space $C_k(X,Y)$ is rectifiable and has the strong Pytkeev property at each point. We also prove that for any pointed spaces $(X_n,*_n)$, $n\in\w$, with the strong Pytkeev property their Tychonoff product $\prod_{n\in\w}X_n$ and their small box-product $\cbox_{n\in\w}X_n$ both have the strong Pytkeev property at the distinguished point $(*_n)_{n\in\w}$. We prove that a sequential rectifiable space $X$ has the strong Pytkeev property if and only if $X$ is metrizable or contains a clopen submetrizable $k_\omega$-subspace. A locally precompact topological group is metrizable if and only if it contains a dense subgroup with the strong Pytkeev property.
\end{abstract}
\maketitle

This paper is devoted to a systematic study of the strong Pytkeev property in topological spaces.
The strong Pytkeev property was introduced by Tsaban and Zdomskyy in \cite{TZ} and is a hereditary version of the Pytkeev property introduced by Pytkeev in \cite{Pyt} and studied in \cite{BM}, \cite{FlD}, \cite{Koc}, \cite{MT}, \cite{PP}, \cite{MSak2}, \cite{MSak06},  \cite{ST}. In Section~\ref{s1} of this paper we discuss the interplay between the strong Pytkeev property and other known local properties of topological spaces, in Sections~\ref{s2} detect function spaces with the strong Pytkeev property, in Section~\ref{s3} we apply the results of Section~\ref{s2} to establish the stability the class of topological spaces with the strong Pytkeev under countable Tychonoff products and small box-products. Section~\ref{s4} is devoted to rectifiable spaces and their topologo-algebraic counterparts called topological lops.  In that section we prove that the strong Pytkeev property is preserved by taking function spaces with values in rectifiable spaces. It should be mentioned that the class of rectifiable spaces contains all spaces homeomorphic to topological groups or topological loops. In Section~\ref{s5} we characterize (sequential) rectifiable spaces with the strong Pytkeev property and in final Section~\ref{s6} we establish some metrizability criteria for countably compact subsets of rectifiable spaces with the strong Pytkeev property.

\section{The interplay between the strong Pytkeev property and other local topological properties}\label{s1}

The central notion of this paper is the strong Pytkeev property introduced by Tsaban and Zdomskyy in \cite{TZ} and studied in \cite{MSak2} and \cite{GKL}.

\begin{definition} A topological space $X$ is defined to have {\em the strong Pytkeev property at} a point $x\in X$ if there is a countable family $\mathcal N$ of subsets of $X$ (called a {\em Pytkeev network at} $x$) such that for any neighborhood $O_x\subset X$ of $x$ and any subset $A\subset X$ accumulating at $x$ there is a set $N\in\mathcal N$ such that $N\subset O_x$ and the intersection $A\cap N$ is infinite.
\end{definition}

 A subset $A\subset X$ {\em accumulates} at a point $x\in X$ if each neighborhood $O_x\subset X$ contains infinitely many points of the set $A$. If $X$ is a $T_1$-space, then $A$ accumulates at $x$ if and only if $x$ belongs to the closure of the set $A\setminus \{x\}$ in $X$.

We say that a topological space $X$ has {\em the strong Pytkeev property} if it has the strong Pytkeev property at each point $x\in X$.  As we said, the strong Pytkeev property was introduced by Tsaban and Zdomskyy \cite{TZ} and afterward was studied by many authors \cite{MSak2}, \cite{GKL}. It is clear that this property belongs to local topological properties.
So, now we analyze the relation of the strong Pytkeev property to other known local properties, whose definitions we recall now.

\begin{definition} Let $X$ be a topological space and $x\in X$. We say that $X$ has
\begin{itemize}
\item {\em the Pytkeev property} at $x$ if for any subset $A\subset X$ accumulating at $x$ there is a countable family $\mathcal N$ of infinite subsets of $X$ (called a {\em Pytkeev $\pi$-network} in $A$ at $x$) such that each neighborhood $O_x\subset X$ contains some set $N\in\mathcal N$;
\item {\em countable $\cs$-character at} $x$ if there is a countable family $\mathcal N$ of subsets of $X$ (called a {\em $\cs$-network at} $x$) such that for every neighborhood $O_x\subset X$ and sequence convergent to $x$ there is a set $N\in\mathcal N$ contained in the neighborhood $O_x$ and containing all but finitely many points of the sequence;
\item {\em countable $\cs^*$-character at} $x$ if there is a countable family $\mathcal N$ of subsets of $X$ (called a {\em $\cs^*$-network at} $x$) such that for every neighborhood $O_x\subset X$ and sequence convergent to $x$ there is a set $N\in\mathcal N$ contained in the neighborhood $O_x$ and containing infinitely many points of the sequence;
\item {\em countable $k$-character at} $x$ if there is a countable family $\mathcal N$ of subsets of $X$ (called a {\em local $k$-network at} $x$) such that for every neighborhood $O_x\subset X$ there is a neighborhood $U_x\subset X$ of $x$ such that for every compact subset $K\subset U_x$ there is a finite subfamily $\F\subset\mathcal N$ such that $K\subset\bigcup\F\subset O_x$;
\smallskip

\item {\em countable tightness at} $x$ if each subset $A\subset X$ with $a\in \cl_X(A)$ contains a countable subset $B\subset A$ such that $x\in \cl_X(B)$;
\item {\em countable} ({\em open}) {\em fan tightness} at $x$ if for every sequence of (open) sets $A_n\subset X$, $n\in\w$,  with $x\in\bigcap_{n\in\w}\bar A_n$ there are finite sets $F_n\subset A_n$, $n\in\w$, such that each neighborhood $O_x\subset X$ meets infinitely many sets $F_n$, $n\in\w$;
\item {\em Fr\'echet-Urysohn property at} a $x$ if each subset $A\subset X$ with $x\in A$ contains a sequence convergent to $x$.
\end{itemize}
\end{definition}

\begin{remark} The Pytkeev property was introduced by Pytkeev in \cite{Pyt} and afterwards studied by many authors. It is known (see \cite[p.208]{MSak06}) that each Pytkeev space has countable tightness. Spaces with countable $\cs$-character and countable $\cs^*$-character were introduced and studied by Banakh and Zdomskyy in \cite{BZd}. They proved that these two notions are equivalent (observing that for any countable $\cs^*$-network $\mathcal N$ at $x$ the family $\widetilde{\mathcal N}=\big\{\bigcup\F:\F\subset\N,\;|\F|<\w\big\}$ is a countable $\cs$-network at $x$). The countable tightness, and the Fr\'echet-Urysohn properties are classical topological notions. The countable fan  tightness appears naturally in the theory of Function Spaces (see \cite[\S II.2]{Arh}) while its ``open'' modification was introduced by M.~Sakai \cite{MSak2} as the property $(\#)$. The notion of local $k$-network is a modification of a well-known notion of a $k$-network. It was introduced in the initial version of this paper and has been used by Gabriyelyan and K\c akol in \cite{GK1}, \cite{GK2}.
\end{remark}

We recall that a topological space $X$ is
\begin{itemize}
\item {\em sequential} if for each non-closed subset $A\subset X$ there is a sequence $\{a_n\}_{n\in\w}\subset A$ convergent to a point $x\in X\setminus A$;
\item {\em subsequential} if $X$ is a subspace of a sequential $T_1$-space.
\end{itemize}
In \cite{Pyt} Pytkeev proved that each subsequential space is Pytkeev.

The interplay between the strong Pytkeev property and some other topological properties is described in the following diagram:
$$\xymatrix{
\mbox{first countable}\ar@{=>}[r]&\mbox{strong Pytkeev}\ar@{=>}[r]\ar@{=>}[d]&\mbox{countable $k$-character}\ar@{=>}[r]&\mbox{countable $\cs^*$-character}\\
\mbox{subsequential}\ar@{=>}[r]&\mbox{Pytkeev}\ar@{=>}[r]&\mbox{countably tight}.
}
$$

The unique non-trivial implication from the upper line in this diagram is proved in the following proposition.


\begin{proposition}\label{vazlyvo} If a topological space $X$ has the strong Pytkeev property at a point $x\in X$, then $X$ has countable $k$-character at $x$.
\end{proposition}

\begin{proof} Let $\mathcal N$ be a countable Pytkeev network at $x$. We lose no generality assuming that $\mathcal N$ is closed under finite unions and each set $N\in\mathcal N$ containing the point $x$ contains also the intersection $\ddot x$ of all neighborhoods of $x$ in $X$ (observe that $\ddot x=\{x\}$ if $X$ is a $T_1$-space).

First we show that $X$ has countable tightness at $x$. Fix any subset $A\subset X$ with $x\in\bar A$. If $x$ is not an accumulation point of $A$, then there is a neighborhood $O_x\subset X$ with finite intersection $B=O_x\cap A$ and then $B$ is a required countable set with $x\in B$. So, we assume that $x$ is an accumulation point of $A$. Consider the subfamily $\mathcal N(A)=\{N\in\mathcal N:N\cap A\ne\emptyset\}$ and for each $N\in\mathcal N(A)$, fix a point $x_N\in A\cap N$. Then $B=\{x_N:N\in\mathcal N(A)\}$ is a countable set with $x\in \bar B$. Indeed, for every neighborhood $O_x\subset X$ of $x$ there is a set $N\in\mathcal N$ such that $N\subset O_x$ and $N\cap A$ is infinite. In this case $N\in\mathcal N(A)$ and hence $x_N\in N\subset O_x$, so $O_x\cap B\ne\emptyset$ and $x\in \bar B$.
\smallskip

Now we are ready to prove that $\mathcal N$ is a local $k$-network at $x$. Assume conversely that this is not true.
Then there is a neighborhood $U\subset X$ of $x$ such that every neighborhood $V\subset X$ of $x$ contains a compact subset $K_V\subset V$ such that $K_V\not\subset \bigcup\F$ for any finite subfamily $\F\subset\mathcal N_U:=\{N\in\mathcal N:N\subset U\}$.

Observe that for some finite set $F_x\subset X$ the union $F_x\cup\bigcup\mathcal N_U$ is a neighborhood of $x$. Assuming that this is not true, we would conclude that the complement $X\setminus \bigcup\mathcal N_U$ accumulates at $x$ and hence has infinite intersection with some set $N\in\mathcal N_U$ which is not possible. So, $F_x\cup\bigcup_{k\in\w}N_k$ is a neighborhood of $x$ for some finite set $F_x\subset X$, disjoint with the union $\bigcup\mathcal N_U$. Since $\mathcal N$ is a network at $x$, some set $N\in\mathcal N_U$ contains $x$ and by our assumption, contain also the intersection $\ddot x$ of all neighborhoods of $x$ in $X$. Then $\ddot x\subset \bigcup\mathcal N_U$ and hence $F_x\cap\ddot x=\emptyset$. Then we can find a neighborhood $W$ of $x$ in $X$, which is disjoint with the finite set $F_x$. Replacing $W$ by a smaller neighborhood of $x$, we can assume that $W\subset F_x\cup\bigcup\mathcal N_U$, which implies that $W\subset \bigcup\mathcal N_U$ and hence $\bigcup\mathcal N_U$ is a neighborhood of $x$.

 Let $\{N_k\}_{k\in\w}$ be an enumeration of the countable family $\mathcal N_U$.
 By our assumption, for every neighborhood $V\subset W$ of $x$ and every $k\in\w$ the compact set $K_V$ contains a point $x_{V,k}\notin \bigcup_{i\le k}{N_i}$. Since $K_V\subset V\subset W\subset\bigcup_{i\in\w}N_i$, the set $\{x_{V,k}\}_{k\in\w}$ is infinite. By the compactness of $K_V$ this set has an accumulation point $x_V\in K_V\subset V$. Let $\mathcal V$ denote the family of all open neighborhoods $V\subset W$ of the point $x$. It follows that the set $\{x_V:V\in\mathcal V\}$ contains the point $x$ in its closure. Since $X$ has countable tightness at $x$, there is a countable subfamily $\{V_n\}_{n\in\w}\subset\mathcal V$ such that $x$ belongs to the closure of the countable set $\{x_{V_n}\}_{n\in\w}\subset X\setminus \{x\}$. Since each point $x_{V_n}$ is an accumulation point of the set $\{x_{V_n,k}\}_{k\ge n}$, the set $A=\{x_{V_n,k}:n\in\w,\;k\ge n\}\subset X\setminus\{x\}$ accumulates at the point $x$. Since $\mathcal N$ is a Pytkeev network at $x$, there is a set $N\in\mathcal N_U$ that has infinite intersection with the set $A$. It follows that $N=N_k$ for some $k\in\w$ and then $N\cap A\subset \{x_{V_n,i}:n\le i<k\}$ is finite. This is a desired contradiction showing that $\mathcal N$ is a local $k$-network at $x$.
\end{proof}

The following simple example shows that the strong Pytkeev property is strictly stronger that the countability of $k$-character.

\begin{example} For any free ultrafilter $p$ on $\IN$ the space $X=\IN\cup\{p\}\subset\beta\IN$
\begin{itemize}
\item has a countable local $k$-network at $p$;
\item fails to have the (strong) Pytkeev property at $p$.
\end{itemize}
\end{example}

\begin{proof} Since each compact subset of the space $X=\IN\cup\{p\}$ is finite, the family of singletons is a countable local $k$-network at $p$.

Next, we show that $X$ does not have the Pytkeev property at the point $p$. It suffices to show that no countable family $\mathcal P=\{P_n\}_{n\in\w}$ of infinite subsets of $X$ is a Pytkeev $\pi$-network at $p$. For every $n\in\w$ by induction we can choose two distinct points $a_n,b_n\in P_n\cap \w\setminus\{ a_k,b_k:k<n\}$. Enlarge the sets $\{a_n\}_{n\in\w}$ and $\{b_n\}_{n\in\w}$ to two disjoint sets $A,B$ such that $A\cup B=\w$. One of the sets $A$ or $B$ belongs to the ultrafilter $p$. If $A\in p$, then  the neighborhood $O_p=\{p\}\cup A$ contains no infinite set $P\in\mathcal P$. If $B\in p$, then the neighborhood $O_p=\{p\}\cup B$ contains no infinite set $P\in\mathcal P$. In both cases we conclude that $\mathcal P$ is not a Pytkeev $\pi$-network at $p$. Consequently, $X$ does not have the Pytkeev property at $p$.
\end{proof}

The strong Pytkeev property combined with the countable fan tightness can be used to characterize first countable spaces.

\begin{proposition}\label{first} A topological space $X$ is first countable at a point $x\in X$ if and only if $X$ has the strong Pytkeev property at $x$ and has countable fan tightness at $x$.
\end{proposition}

\begin{proof}  The ``only if'' part is trivial. To prove the ``if'' part, assume that a topological space $X$ with countable fan tightness at $x$ has a countable Pytkeev network $\mathcal N$ at $x$.
We lose no generality assuming that $\mathcal N$ is closed under finite unions and each set $N\in\mathcal N$ containing $x$ contains also the intersection $\ddot x$ of all neighborhoods of $x$ in $X$.

For any set $N\in\mathcal N$ by $N^\circ$ we denote its interior in $X$. We claim that the countable family ${\mathcal N}^\circ=\{N^\circ:N\in\mathcal N,\;x\in N^\circ\}$ is a neighborhood base at $x$. Given any neighborhood $U\subset X$ of $x$, we need to find a set $N\in\mathcal N$ with $x\in N^\circ\subset U$. Consider the countable subfamily $\mathcal N_U=\{N\in\mathcal N\colon x\in N\subset U\}$ and let $\mathcal N_U=\{N_k\}_{k\in\w}$ be its enumeration. We claim that the set  $C=X\setminus\bigcup\mathcal N_U$ does not accumulate at $x$. In the opposite case there would exist a set $N\in\mathcal N$ such that $N\subset U$ and $N\cap C$ is infinite, which is not possible (as $N\in\mathcal N_U$ and hence $N\cap C\subset (\bigcup\mathcal N_U)\cap C=\emptyset$). Therefore $x$ is not an accumulation point of $C$. In this case we can find a neighborhood $W\subset U$ of $x$ such that the set $W\cap C$ is finite. Since $C\cap\ddot x=\emptyset$, we can replace $W$ by a smaller neighborhood and assume that $W\cap C$ is empty, which means $W\subset \bigcup\mathcal N_U$.
\smallskip

Assuming that $x\notin N^\circ$ for every $N\in\mathcal N_U$ and taking into account that the family $\mathcal N_U$ is closed under finite unions, we conclude that the every $k\in\w$ the set $A_k=X\setminus \bigcup_{i\le k}N_i$ contains the point $x$ in its closure.
By the countable fan tightness of $X$ at $x$, there are finite sets $F_k\subset A_k\cap W$, $k\in\w$, such that every neighborhood $O_x$ of $x$ meets infinitely many sets $F_k$, $k\in\w$. Since $W\subset \bigcup_{k\in\w}N_k$, for every $k\in\w$ there is $n\in\w$ such that $F_k\subset \bigcup_{i\le n}N_i$ and hence $F_k\cap F_m=\emptyset$ for all $m\ge n$. This observation implies that $x$ is an accumulation point of the set $A=\bigcup_{k\in\w}F_k$.

Since $\mathcal N$ is a Pytkeev network at $x$, there is a set $N\in\mathcal N$ such that $x\in N\subset U$ and $N\cap A$ is infinite. The set $N$ belongs to the family $\mathcal N_U$ and hence coincides with some set $N_m$. The choice of the sequence $(F_k)$ guarantees that $N\cap A=N\cap\bigcup_{k\in\w}F_k\subset \bigcup_{k<m}F_k$ is finite, which is a desired contradiction showing that $x\in N^\circ\subset N\subset U$ for some $N\in\mathcal N$. So, $\mathcal N^\circ=\{N^\circ:N\in\mathcal N\}$ is a countable neighborhood base at $x$.
\end{proof}

For regular topological spaces the countable fan tightness in Proposition~\ref{first} can be weakened to the countable fan open-tightness (as was observed by  M.~Sakai in \cite[3.2]{MSak2}).

\begin{proposition}[Sakai]\label{first:sak} A regular topological space $X$ is first countable at a point $x\in X$ if and only if $X$ has both the strong Pytkeev property and the countable fan open-tightness at $x$.
\end{proposition}

The following proposition gives conditions under which the strong Pytkeev property is equivalent to some weaker properties (countable $k$-character or countable $\cs^*$-character). Having in mind the notion of a $k$-space we define a topological space $X$ to be a {\em $k_2$-space} if for any non-closed subset $A\subset X$ there exists compact Hausdorff subspace $K\subset X$ such that $A\cap K$ is not closed in $K$. It is clear that each $k_2$-space is a $k$-space and each Hausdorff $k$-space $X$ is a $k_2$-space. So, in the realm of Hausdorff spaces these two notions coincide.

\begin{proposition} A topological space $X$ has the strong Pytkeev property at a point $x\in X$ if one of the following conditions is satisfied:
\begin{enumerate}
\item $X$ has countable $\cs^*$-character and is Fr\'echet-Urysohn at $x$;
\item $X$ has countable $k$-character at $x$ and is sequential or a $k_2$-space.
\end{enumerate}
\end{proposition}

\begin{proof} 1. Assume that $X$ is Fr\'echet-Urysohn and has countable $\cs^*$-character at $x$. Then $X$ has a countable $\cs^*$-network $\N$ at $x$. We claim that $\N$ is a Pytkeev network at $x$. Given a neighborhood $O_x\subset X$ of $x$ and a subset $A\subset X$ accumulating at $x$, we need to find
a subset $N\in\mathcal N$ of $O_x$ having infinite intersection with  $A$.

Let $\ddot x$ be the intersection of all neighborhoods of $x$. If $\ddot x\cap A$ is infinite, then we can choose any sequence $(a_n)_{n\in\w}$ of pairwise distinct points of $\ddot x\cap A$ and conclude that this sequence converges to $x$. By definition, the $\cs^*$-network $\mathcal N$ contains a set $N\subset O_x$ containing infinitely many points of the sequence $(a_n)$ and hence infinitely many points of the set $A$.

So, we assume that $\ddot x\cap A$ is finite. Then $x$ remains an accumulation point of the set $A\setminus \ddot x$ and by the Fr\'echet-Urysohn property of $X$ at $x$, we can find a sequence $(a_n)_{n\in\w}$ in $A\setminus \ddot x$, convergent to $x$. It follows that this sequence contains no constant subsequences, which allows us to select a subsequence in $(a_n)_{n\in\w}$ consisting of pairwise distinct points. Replacing the sequence $(a_n)$ by this subsequence, we can assume that the points $a_n$, $n\in\w$, are pairwise distinct. By definition, the $\cs^*$-network $\mathcal N$ contain a set $N\subset O_x$ containing infinitely many points of the sequence  $(a_n)_{n\in\w}$ and hence infinitely many points of the set $A$.
\smallskip

2. Assume that $X$ has countable $k$-character and is sequential or a $k_2$-space. Then $X$ admits a countable local $k$-network $\mathcal N$ at $x$. We claim that $\mathcal N$ is a Pytkeev network at $x$. Given a neighborhood $O_x\subset X$ of $x$ and a subset $A\subset X$ accumulating at $x$, we need to find a set $N\in\mathcal N$ such that $N\subset O_x$ and $N\cap A$ is infinite.

By the definition of the local $k$-network $\mathcal N$, there exists an open neighborhood $U_x\subset O_x$ such that each compact subset $K\subset U_x$ is covered by finitely many sets of the network $\mathcal N$ that are contained in $O_x$. Let $\ddot x$ be the intersection of all neighborhoods of the point $x$. It is clear that $\ddot x$ is a compact subset of $U_x$, so $\ddot x\subset\bigcup\F\subset O_x$ for some finite subfamily $\F\subset \mathcal N$. If the intersection $\ddot x\cap A$ is infinite, then for some $F\in\F\subset\mathcal N$ the intersection $A\cap F$ is infinite and we are done.

So, we assume that $\ddot x\cap A$ is finite. Since $x$ is an accumulating point of $A$, it is also accumulating point of $A\setminus \ddot x$. Replacing the set $A$ by $A\setminus \ddot x$ we can assume that
$A\cap \ddot x=\emptyset$.
It follows that the set $B=A\cup (X\setminus U_x)$ accumulates at $x$ and hence it is not closed in $X$.

If $X$ is a $k_2$-space, then we can find a compact Hausdorff subspace $K\subset X$ such that $K\cap B$ is not closed in $K$. Consequently, there exists a point $z\in K\setminus B\subset U_x$ which belongs to the closure of the set $K\cap B$. The compact space $K$, being Hausdorff, is regular. So, we can choose a compact neighborhood $K_z\subset K\cap U_x$ of the point $z$ in $K$. Observe that the set $K_z\setminus B$ is infinite (since $z\notin K_z\setminus B$ is contained in the closure of $K_z\setminus B$ and $K_z$ is Hausdorff). By the choice of the neighborhood $U_x$, for the compact set $K_z\subset U_x$ there is a finite subfamily $\F\subset\mathcal N$ such that $K_z\setminus B\subset K_z\subset\bigcup\F\subset O_x$. By Pigeonhole Principle, for some set $F\in\F\subset\mathcal N$ the set $F\setminus B=F\cap A$ is infinite. This completes the proof of the $k_2$-case.
\smallskip

Now assume that the space $X$ is sequential. By our assumption, the set $A$ misses $\ddot x$ and contains the point $x$ in its closure. Then the set $B=A\cup(X\setminus U_x)$ contains $x$ in its closure too and does not intersect $\ddot x$ (as $\ddot x\subset U_x$). For any subset $C\subset X$ denote by $C^{(1)}$ the sequential closure of $C$, i.e., the set of limit points of sequences $\{c_n\}_{n\in\IN}\subset C$ that converge in $X$. Put $B^{(0)}=B$ and by transfinite induction, for every ordinal $\alpha>1$ put $B^{(\alpha)}=(B^{(<\alpha)})^{(1)}$ where $B^{(<\alpha)}=\bigcup_{\beta<\alpha}B^{(\beta)}$. The sequentiality of the space $X$ guarantees that $x\in B^{(\alpha_x)}$ for some ordinal $\alpha_x>0$. If $x\in B^{(1)}$, then we can choose a sequence $\{b_n\}_{n\in\w}\subset B$, convergent to $x$.
Since $B\cap \ddot x=\emptyset$, this sequence cannot contain a constant subsequence and hence it contains a subsequence of pairwise distinct points of $B$. Replacing the sequence $(b_n)_{n\in\w}$ by this subsequence, we can assume that all points $b_n$, $n\in\w$, are pairwise distinct and belong to the neighborhood $U_x$. The local $k$-network $\mathcal N$, being a $\cs^*$-network at $x$, contains a set $N\subset O_x$ containing infinitely many points of the convergent sequence $\{b_n\}_{n\in\w}\subset B\cap U_x\subset A$ and hence infinitely many points of the set $A$. Now assume that $x\notin B^{(1)}$.

A sequence of points $(b_n)_{n\in\w}$ will be called {\em injective} if $b_n\ne b_m$ for any distinct numbers $n,m\in\w$.

\begin{claim}\label{cl1.8} For every ordinal $\alpha\ge 1$, any open neighborhood $O_z\subset X$ of any point $z\in B^{(\alpha)}\setminus B^{(1)}$ contains an injective sequence $\{b_n\}_{n\in\w}\subset B$ convergent to a point of $O_z$.
\end{claim}

\begin{proof} This claim will be proved by induction on $\alpha\ge 1$.
For $\alpha=1$ this statement trivially holds. Assume that for some ordinal $\alpha>1$ the claim has been proved for all ordinals $<\alpha$. Choose any point $z\in B^{(\alpha)}\setminus B^{(1)}$ and an open neighborhood $O_z\subset X$ of $z$. If $z\in B^{(\beta)}$ for some $\beta<\alpha$, then we can apply the inductive assumption and complete the proof.  So, assume that $z\notin B^{(<\alpha)}:=\bigcup_{\beta<\alpha}B^{(\beta)}$. In this case $z$ is the limit of some convergent sequence $\{z_n\}_{n\in\w}\subset B^{(<\alpha)}\cap O_z$. If $\alpha>2$, then $z\notin B^{(<\alpha)}$ implies the existence of $n\in\w$ such that $z_n\notin B^{(1)}$. Then by the inductive assumption, the neighborhood $O_z$ of $z_n$ contains an injective sequence $\{b_n\}_{n\in\w}\subset B$ converging to some point of $O_z$ and we are done.

It remains to consider the case $\alpha=2$. Assume that the neighborhood $O_z$ contains no injective sequence $\{b_n\}_{n\in\w}\subset B$ convergent to a point of $O_z$. Since $z\in B^{(2)}\setminus B^{(1)}$, in the sequence $(z_n)_{n\in\w}$ there is only finitely many points of the set $B$. Passing to a subsequence, we can assume that $z_n\notin B$ for all $n\in\w$. Since $O_z$ contains no injective convergent sequence of points of $B$, for every $n\in\w$ the point $z_n\in O_z\cap B$ is the limit of some constant sequence in $B$, which implies that the intersection $\ddot{z}_n$ of all neighborhoods of $z_n$ meets the set $B$ and hence contains a point $b_n\in B\cap \ddot{z}_n$. It is easy to see that the convergence of the sequence $(z_n)_{n\in\w}$ to $z$ implies the convergence of $(b_n)_{n\in\w}$ to the same limit $z$. This means that $z\in B^{(1)}$, which contradicts the choice of $z$.
This contradiction completes the proof of Claim~\ref{cl1.8}.
\end{proof}

By Claim~\ref{cl1.8}, for the ordinal $\alpha_x$ the neighborhood $U_x$ of the point $x\in B^{(\alpha_x)}\setminus B^{(1)}$ contains an injective sequence $\{b_n\}_{n\in\w}\subset B$ convergent to a point $b_\infty\in U_x$. Then $K=\{b_\infty\}\cup\{b_n:n\in\w\}$ is an infinite compact subset of $U_x$. By the choice of $U_x$, there is a finite subfamily $\F\subset\mathcal N$ such that $K\subset \bigcup\F\subset O_x$. By the Pigeonhole Principle, some set $F\in\F\subset\mathcal N$ contains infinitely many points of the sequence $\{b_n\}_{n\in\w}\subset B\cap U_x\subset A$ and hence has infinite intersection $F\cap A$ with the set $A$. This witnesses that $\mathcal N$ is a Pytkeev network at $x$.
\end{proof}

In \cite{FlD} Fedeli and le Donne constructed an example of a Tychonoff space with the Pytkeev property, which is not subsequential. We shall see that this space has the strong Pytkeev property.

\begin{example}[Fedeli, le Donne] There exists a countable regular space $X$ with a unique non-isolated point which has the strong Pytkeev property but is not subsequential.
\end{example}

\begin{proof} We recall the construction of \cite{FlD}. Fix any free filter $\F$ on the set $\w$. Take any point $\infty\notin\w\times\w$ and consider the space $X=(\w\times\w)\cup\{\infty\}$ endowed with the topology in which all points of the set $\w\times\w$ are isolated while the sets
$$O_{F,\varphi}=\{\infty\}\cup\{(x,y)\in F\times\w:y\ge \varphi(x)\}\mbox{ \ for $F\in\F$ and  $\varphi\in\w^\w$}$$
form a neighborhood base at the point $\infty$. For every points $x,y\in\w$ put $N_{x,y}=\{(u,v)\in\w\times\w:u=x,\;v\ge y\}$ and observe that $\mathcal N=\{N_{x,y}:x,y\in\w\}$ is a countable Pytkeev network at $\infty$, which means that $X$ has the strong Pytkeev property.
By \cite{FlD}, the space $X$ is not subsequential.
\end{proof}

\begin{remark}
The results of this section fit into the following diagram:

$$
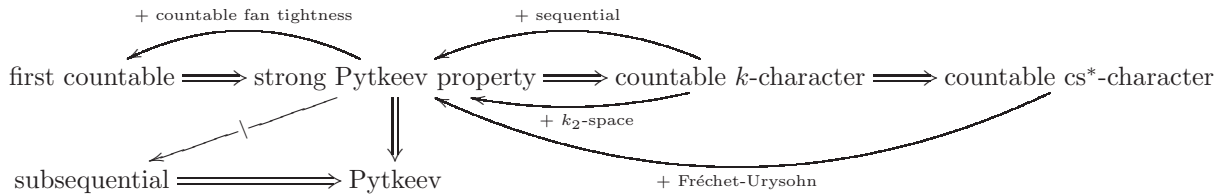
\xymatrix{
\mbox{first countable}\ar@{=>}[r]&\mbox{strong Pytkeev property}\ar@{->}|-{\backslash}[dl]\ar@{=>}[r]\ar@{=>}[d]\ar@/_1.5pc/_{\mbox{\tiny + countable fan  tightness}}[l]&\mbox{countable $k$-character}\ar@{=>}[r]\ar@/^0.9pc/^(0.45){\mbox{\tiny + $k_2$-space}}[l]\ar@/_1.5pc/_{\mbox{\tiny + sequential}}[l]&\mbox{countable $\cs^*$-character}\ar@/^2.8pc/^{\mbox{\tiny + Fr\'echet-Urysohn}}[ll]\\
\mbox{subsequential}\ar@{=>}[r]&\mbox{Pytkeev}
}
$$
\end{remark}
\smallskip

\section{Function spaces with compact-open topology}\label{s2}

In this section we shall establish the strong Pytkeev property in spaces $C_\I(X,Y)$ of continuous functions endowed with the $\I$-open topology, or else the topology of uniform convergence of sets belonging to some ideal $\I$ of compact subsets.

Let $X$ be a topological space. A family $\I$ of compact subsets of $X$ is called {\em an ideal of compact sets} if $\bigcup\I=X$ and for any sets $A,B\in\I$ and compact subset $K\subset X$ we get $A\cup B\in\I$ and $A\cap K\in\I$.

For an ideal $\I$ of compact subsets of a topological space $X$ and a topological space $Y$ by $C_\I(X,Y)$ we shall denote the space $C(X,Y)$ of all continuous functions from $X$ to $Y$, endowed with the {\em $\I$-open topology} generated by the subbase consisting of the sets
$$[K;U]=\{f\in C_\I(X,Y):f(K)\subset U\}$$where $K\in\I$ and $U$ is an open subset of $Y$.

If $\I$ is the ideal of compact (resp. finite) subsets of $X$, then the $\I$-open topology coincides with the compact-open topology (resp. the topology of pointwise convergence) on $C(X,Y)$. In this case the function space $C_\I(X,Y)$ will be denoted by $C_k(X,Y)$ (resp. $C_p(X,Y)$).

We shall be interested in detecting function spaces $C_\I(X,Y)$ possessing the strong Pytkeev property. For this we should impose some restrictions on the ideal $\I$.

\begin{definition} An ideal $\I$ of compact subsets of a topological space $X$ is defined to be {\em discretely-complete} if for any compact subset $A,B\subset X$ such that $A\setminus B$ is a countable discrete subspace of $X$ the inclusion $B\in\I$ implies $A\in\I$.
\end{definition}

It is clear that the ideal of all compact subsets of $X$ is discretely-complete. More generally, for any infinite cardinal $\kappa$ the ideal $\I$ of compact subsets of cardinality $\le\kappa$ in  $X$ is discretely-complete. On the other hand, the ideal of finite subsets of $X$ is discretely-complete if and only if $X$ contains no infinite compact subset with finite set of non-isolated points.

Let us recall \cite[\S11]{Gru} that a family $\mathcal N$ of subsets of a topological space is a {\em $k$-network} in $X$ if for any open set $U\subset X$ and compact subset $K\subset U$ there is a finite subfamily $\F\subset \mathcal N$ such that $K\subset\bigcup\F\subset U$.
It is easy to see that a family $\mathcal N$ of a (Hausdorff) space $X$ is a $k$-network in $X$ (if and) only if $X$ is a local $k$-network at every point $x\in X$.

Regular topological spaces with countable $k$-network are called {\em $\aleph_0$-spaces}. Such spaces were introduced by E.~Michael \cite{Mi} who proved that for any $\aleph_0$-spaces $X,Y$ the function space $C_k(X,Y)$ is an $\aleph_0$-space. In particular, for any separable metrizable spaces $X,Y$ the function space $C_k(X,Y)$ is an $\aleph_0$-space. In \cite{Ban} this Michael's results was improved  to $\Pyt_0$-spaces. Following \cite{Ban} we define a regular topological space $X$ to be a {\em $\Pyt_0$-space} if it has a {\em countable Pytkeev network}, i.e., a countable network $\mathcal N$, which is a Pytkeev network at each point $x\in X$. By \cite{Ban}, each $\Pyt_0$-space is an $\aleph_0$-space, and for any $\aleph_0$-space $X$ and a $\Pyt_0$-space $Y$ the function space $C_k(X,Y)$ is a $\Pyt_0$-space. Taking into account that each $\Pyt_0$-space has the strong Pytkeev property, we can regard the following theorem as a local version of the mentioned result of \cite{Ban}.

\begin{theorem}\label{main} Let be an $\aleph_0$-space (more generally, a Hausdorff space with countable $k$-network) and  $\I$ be a discretely-complete ideal of compact subsets of $X$. If a topological space $Y$ has the strong Pytkeev property at a point $y_0\in Y$, then the function space $C_\I(X,Y)$ has the strong Pytkeev property at the constant function $f:X\to\{y_0\}\subset Y$.
\end{theorem}

\begin{proof}
Let $\mathcal K$ be a countable $k$-network on the space $X$ and $\mathcal P$ be a Pytkeev network at the point $y_0\in Y$ of the space $Y$. We lose no generality assuming that the networks $\mathcal K$ and $\mathcal P$ are closed under finite unions and finite intersections and each set $P\in\mathcal P$ contains the point $y_0$.

For two subsets $K\subset X$ and $P\subset Y$ let
$$[K;P]=\{f\in C_\I(X,Y):f(K)\subset P\}\subset C_\I(X,Y).$$
We claim that the countable family
$$[\kern-1.5pt[\mathcal K;\mathcal P]\kern-1.5pt]=\{[K;P]:K\in\K,\;P\in\mathcal P\}$$
is a Pytkeev network at the constant function $f:X\to\{y_0\}\subset Y$ in the function space $C_\I(X,Y)$.

Given a subset $\A\subset C_\I(X,Y)$ accumulating at $f$ and a neighborhood $O_f\subset C_\I(X,Y)$ of $f$ we need to find a set $\F\in[\kern-1.5pt[\K,\PP]\kern-1.5pt]$ such that $\F\subset O_f$ and $\A\cap \F$ is infinite.

We lose no generality assuming that $\A\subset O_f$ and the neighborhood $O_f$ is of basic form
$O_f=[C;U]$
for some compact set $C\in\I$ and some open neighborhood $U\subset Y$ of the point $y_0$.

Let $\{K'_{j}\}_{j\in\w}$ be an enumeration of the countable subfamily $\K(C)=\{K\in\K:C\subset K\}\subset\K$. For every $j\in\w$ let
$K_{j}=\bigcap_{i\le j}K'_{i}$. It follows that the decreasing sequence of sets  $(K_{j})_{j=1}^\infty$ converges to $C$ in the sense that each neighborhood $O_C\subset X$ of $C$ contains all but finitely many sets $K_{j}$.

Consider the countable subfamily $\mathcal P(U)=\{P\in\mathcal P:P\subset U\}$ and let $\mathcal P(U)=\{P_j'\}_{j\in\w}$ be its enumeration. For every $j\in\w$ let $P_{j}=\bigcup_{i\le j}P_{i}'$. It follows that $(P_{j})_{j\in\w}$ is an increasing sequence of sets.

Then the sets
$$[K_{j};P_{j}]\in[\kern-1.5pt[\mathcal K;\mathcal P]\kern-1.5pt]$$form an increasing sequence of sets in the function space $C_{c}(X,Y)$.

By (the proof of) Proposition~\ref{vazlyvo}, the countable Pytkeev network $\mathcal N$ at $y_0$ is a local $k$-network at $y_0$.
Consequently, there exists an open neighborhood $V\subset U$ of $x$ such that each compact subset $K\subset V$ is contained in the union $P_j=\bigcup_{i\le j}P_i'$ for some $j\in\w$.

\begin{claim}\label{cl1} $[C;V]\subset \bigcup_{j\in\w}[K_j;P_j]$.
\end{claim}

\begin{proof} Assuming the opposite, we can find a function $g\in [C;V]$ such that for every $j\in\w$ we get $g\notin [K_j;P_j]$, which yields a point $x_j\in K_j$ with $g(x_j)\notin P_j$.

Since the sequence of sets $(K_j)_{j\in\w}$ converges to $C$, there is a number $j'\in\w$ such that $K_j\subset g^{-1}(V)$ for all $j\ge j'$. Moreover, the set $C'=C\cup\{x_j:j\ge j'\}$ is compact. Since $X$ is a Hausdorff space with countable network, the compact set $C'$ is metrizable and hence the sequence $(x_j)_{j\ge j'}$ contains a subsequence $(x_{j_i})_{i\in\w}$ convergent to some point $x_\infty\in C$. By the continuity of $g$, the sequence $(g(x_{j_i}))_{i\in\w}$ converges to the point $g(x_\infty)\in g(C)\subset V$. Then the subset $K=\{g(x_\infty)\}\cup\{g(x_{j_i})\}_{i\in\w}$ of $V$ is compact and by the choice of the neighborhood $V$, it is contained in some set $P_j=\bigcup_{i\le j}P_{i}$, which is not possible as $g(x_{j_i})\notin P_j$ for $j_i>j$. This contradiction completes the proof of the claim.
\end{proof}

Replacing the set $\A$ by the intersection $\A\cap[C;V]$, we can assume that $\A\subset[C;V]\subset\bigcup_{j\in\w}[K_j;P_j]$. The following claim completes the proof of Theorem~\ref{main}.

\begin{claim}\label{cl2} If $f$ is an accumulation point of $\A$, then  for some $j\in\w$ the intersection $[K_j;P_j]\cap \A$ is infinite.
\end{claim}

\begin{proof} Assume conversely that for every $j\in\w$ the intersection $\A_j=[K_j;P_j]\cap\A$ is finite.
Then $\A=\bigcup_{j\in\w}\A\cap[K_j;P_j]=\bigcup_{j\in\w}\A_j$ is the countable union of an increasing sequence $(\A_j)_{j\in\w}$ of finite subsets of $C_\I(X,Y)$.

To each function $\alpha\in\A\setminus\A_0$ assigns a unique number $j_\alpha\in\w$ such that $\alpha\in\A_{j_\alpha+1}\setminus\A_{j_\alpha}=\A_{j_\alpha+1}\setminus[K_{j_\alpha};P_{j_\alpha}]$ and hence $\alpha(x_\alpha)\notin P_{j_\alpha}$ for some point $x_\alpha\in K_{j_\alpha}$.
It follows that the number sequence $(j_\alpha)_{\alpha\in\A\setminus\A_0}$ tends to infinity in the sense that for any number $l\in\w$ the set $\{\alpha\in\A\setminus\A_0:j_\alpha\le l\}$ is finite.

The convergence of the sequence $(K_j)_{j\in\w}$ to $C$ implies that the set $C'=C\cup\{x_\alpha:\alpha\in\A\setminus\A_0\}$ is compact. The discrete completeness of the ideal $\I$ guarantees that the compact set $C'$ belongs to $\I$ (as $C'\setminus C\subset \{x_\alpha\in\A\setminus\A_0\}$ is a countable discrete subspace of $X$).

We claim that the point $y_0$ is not accumulation point of the set $B=\{\alpha(x_\alpha):\alpha\in\A\setminus\A_0\}\subset Y\setminus\{y_0\}$. Otherwise, there would exist a set $P_j'\in\mathcal P(U)$ such
$P'_j\cap B$ is infinite, which is not possible as $P'_j\cap\{\alpha(x_\alpha):\alpha\in\A\setminus\A_0\}\subset\{\alpha(x_\alpha):\alpha\in\A_j\setminus\A_0\}$ is finite. Since $y_0$ fails to be an accumulation point of the set $B$, there exists an open neighborhood $W\subset V$ of $y_0$ in $Y$ such that $W\cap B$ is finite. Then the set $[C';W]$ is an open neighborhood of the constant function $f$ in the function space $C_\I(X,Y)$ which has finite intersection with the set $\A\setminus\A_0$. But this contradicts our assumption that $f$ is an accumulation point of $\A$.
\end{proof}
\end{proof}

\begin{remark} In \cite{MSak2} M.~Sakai proved that for any uncountable Tychonoff space $X$ the function space $C_p(X,\IR)$ has uncountable $\cs^*$-character and hence fails to have the strong Pytkeev property. This shows that the discrete-completeness of the ideal is an essential requirement in Theorem~\ref{main}.
\end{remark}

Since the ideal of all compact subsets of a space $X$ is discretely-complete, Theorem~\ref{main} implies:

\begin{corollary}\label{c:main}  Let be an $\aleph_0$-space (more generally, a Hausdorff space with countable $k$-network). If a topological space $Y$ has the strong Pytkeev property at a point $e\in Y$, then the function space $C_k(X,Y)$ has the strong Pytkeev property at the constant function $f:X\to\{e\}\subset Y$.
\end{corollary}

At presence of topological homogeneity of the function space $C_\I(X,Y)$, Theorem~\ref{main} allows to establish the Pytkeev property at each point of $C_\I(X,Y)$ (not only at a constant function).
Let us recall that a topological space $X$ is {\em topologically homogeneous} if for any points $x,y\in X$ there is a homeomorphism $h:X\to X$ such that $h(x)=y$.

\begin{corollary}\label{c:homoPyt} Let $X$ be an $\aleph_0$-space, $\I$ be a discretely-complete ideal of compact subsets of $X$, and $Y$ be a topological space with the strong Pytkeev property at some point $e\in Y$. If the function space $C_\I(X,Y)$ is topologically homogeneous, then $C_\I(X,Y)$ has the strong Pytkeev property.
\end{corollary}

This corollary motivates the problem of detecting spaces $Y$ for which the function space $C_\I(X,Y)$ is topologically homogeneous for any $\aleph_0$-space $X$ endowed with a  discretely-complete ideal $\I$ of compact subsets. Taking into account that for any topological group $Y$ the space $C_\I(X,Y)$ is topological group (cf. Proposition~\ref{p:magmaC}), we get:

\begin{corollary}\label{c:groupPyt} Let $X$ be an $\aleph_0$-space and $\I$ be a discretely complete ideal of compact subsets of $X$. For any topological group $Y$ with the strong Pytkeev property, the
function space $C_\I(X,Y)$ has the strong Pytkeev property.
\end{corollary}

Since the ideal of all compact subset of any topological space is discretely-complete Corollary~\ref{c:groupPyt} implies:

\begin{corollary}\label{c:groupPytc} For any $\aleph_0$-space $X$ and topological group $Y$ with the strong Pytkeev property, the
function space $C_k(X,Y)$ has the strong Pytkeev property.
\end{corollary}

In Section~\ref{s4} we shall prove that Corollaries~\ref{c:groupPyt} and \ref{c:groupPytc} remains true for any rectifiable space $Y$.

\section{Preserving the strong Pytkeev property by Tychonoff and small box-products}\label{s3}

In this section we shall detect some topological operations preserving the class of topological spaces with the strong Pytkeev property. We recall that a {\em pointed space} is a topological space $X$ with a distinguished point, which will be denoted by $*_X$.

By the {\em Tychonoff product} of pointed topological spaces $(X_\alpha,*_{X_\alpha})$ we understand the Tychonoff product $\prod_{\alpha\in A}X_\alpha$ with the distinguished point $(*_{X_\alpha})_{\alpha\in A}$. The {\em box-product} $\square_{\alpha\in A}X_\alpha$ of the spaces $X_\alpha$, $\alpha\in A$, is their Cartesian product $\prod_{\alpha\in A}X_\alpha$ endowed with the box-topology generated by the products $\prod_{\alpha\in A}U_\alpha$ of open sets $U_\alpha\subset X_\alpha$, $\alpha\in A$.

The subset $$\cbox_{\alpha\in A}X_\alpha=\big\{(x_\alpha)_{\alpha\in A}\in\square_{\alpha\in A}X_\alpha:\{\alpha\in A:x_\alpha\ne *_{X_\alpha}\}\mbox{ is finite}\big\}$$of the box-product $\square_{\alpha\in A}X_\alpha$ is called the {\em small box-product} of the pointed topological spaces $X_\alpha$, $\alpha\in A$. It is a pointed topological space with distinguished point $(*_{X_\alpha})_{\alpha\in A}$.

The subspace
$$\coprod_{\alpha\in A}X_\alpha =\big\{(x_\alpha)_{\alpha\in A}\in\prod_{\alpha\in A}X_\alpha:\big|\{\alpha\in A:x_\alpha\ne *_{X_\alpha}\}\big|\le1\big\}$$
of the Tychonoff product $\coprod_{\alpha\in A}X_n$ is called the {\em Tychonoff bouquet} of the pointed spaces $X_\alpha$, $\alpha\in A$. The same set $\coprod_{\alpha\in A}X_\alpha$ endowed with the box topology inherited from $\square_{\alpha\in A}X_\alpha$ will be called the {\em box-bouquet} of the pointed topological spaces $X_\alpha$, $\alpha\in A$, and will be denoted by $\bigvee_{\alpha\in A}X_\alpha$.

\begin{theorem} If $X_n$, $n\in\w$, are pointed topological spaces with the strong Pytkeev property at their distinguished points, then
the Tychonoff bouquet $\coprod_{n\in\w}X_n$ and the Tychonoff product $\prod_{n\in\w}X_n$ both have the strong Pytkeev property at their distinguished point $(*_{X_n})_{n\in\w}$.
\end{theorem}

\begin{proof} For every $n\in\w$, fix a countable Pytkeev network $\mathcal N_n$ at the distinguished point $*_{X_n}$ of the space $X_n$. Identify each space $X_n$, $n\in\w$, with the subspace
 $\big\{(x_k)_{k\in\w}\in\coprod_{k\in\w}X_k:\forall k\ne n\;\; (x_k=*_{X_k})\big\}$ of the Tychonoff bouquet $X=\coprod_{k\in\w}X_k$  and observe that the family $$\mathcal N=\big(\bigcup_{n\in\w}\mathcal N_k\big)\cup \big\{\textstyle{\bigcup_{k\ge n}X_k:n\in\w}\big\}$$ is a countable Pytkeev network at the distinguished point $*_X$ of $X$, which means that the space $X$ has the strong Pytkeev property at $*_X$.

Observe that the Tychonoff product $\prod_{n\in\w}X_n$ can be identified with the (pointed) subspace $\{f\in C_k(\w,X):\forall n\in\w\;\;f(n)\in X_n\}$ of the function space $C_k(\w,X)$ whose distinguished point is the constant function $f_*:\w\to\{*_X\}\subset X$. By Corollary~\ref{c:main}, the space $C_k(\w,X)$ has the strong Pytkeev property at its distinguished point $f_*$. Then the pointed subspace $\prod_{n\in\w}X_n$ of  $C_k(\w,X)$ has the strong Pytkeev property at its distinguished point too.
\end{proof}

In the same way we can prove the preservation of the strong Pytkeev property by small box-products.

\begin{proposition} If $X_n$, $n\in\w$, are pointed topological $T_1$-spaces with the strong Pytkeev property at their distinguished points, then
the box-bouquet $\bigvee_{n\in\w}X_n$ and the small box-product $\cbox_{n\in\w}X_n$ both have the strong Pytkeev property at their distinguished points.
\end{proposition}

\begin{proof} For every $n\in\w$ fix a countable Pytkeev network $\mathcal N_n$ at the distinguished point $*_{X_n}$ of the space $X_n$. Identify each space $X_n$ with the subspace
 $\big\{(x_k)_{k\in\w}\in\textstyle{\bigvee_{k\in\w}}X_k:\forall k\ne n\;\; (x_k=*_{X_k})\big\}$ of the box-bouquet $X=\bigvee_{k\in\w}X_k$.
It is easy to check that the countable family $$\mathcal N=\bigcup_{n\in\w}\mathcal N_k$$ is a Pytkeev network at the distinguished point $*_X$ of $X$, which means that $X$ has the strong Pytkeev property at $*_X$.

Let $\w+1=\w\cup\{\infty\}$ be the one-point compactification of the (discrete) space $\w$ of non-negative integers. Observe that the small box-product $\cbox_{n\in\w}X_n$ is homeomorphic to the (pointed) subspace
$$\{f\in C_k(\w+1,X):f(\infty)=*_X\mbox{ and $\forall n\in\w$ $f(n)\in X_n$}\}$$of the function space $C_k(\w+1,X)$.
By Corollary~\ref{c:main}, the function space $C_k(\w+1,X)$ has the strong Pytkeev property at the constant function $f_*:\w\to\{*_X\}$, which implies that the subspace $\cbox_{n\in\w}X_n$ of $C_k(\w+1,X)$ has the strong Pytkeev property at its distinguished point $(*_{X_n})_{n\in\w}$.
\end{proof}

\section{Rectifiable spaces, continuously homogeneous spaces and topological lops}\label{s4}

In this section we return to the problem of detecting topologically homogeneous function spaces $C_\I(X,Y)$, mentioned in Section~\ref{s2}.

Let us recall that a topological space $X$ is {\em topologically homogeneous} if for any points $x,y\in X$ there exists a homeomorphism $h_{x,y}:X\to X$ such that $h_{x,y}(x)\to y$. If the homeomorphism $h_{x,y}$ can be chosen to depend continuously on $x$ and $y$ (in the sense that the map $:X^3\to X^3$, $H:(x,y,z)\mapsto (x,y,h_{x,y}(z))$ is a homeomorphism), then the space $X$ is called {\em continuously homogeneous}. More precisely, a topological space $X$ is continuously homogeneous if there exists a homeomorphism $H:X^3\to X^3$ such that $H(x,y,x)=(x,y,y)$ and $H(\{(x,y)\}\times X)=\{(x,y)\}\times X$ for any points $x,y\in X$. Continuously homogeneous spaces were introduced by Uspenski\u\i\ \cite{Us} (who called them strongly homogeneous spaces). In \cite[Proposition 15]{Us} Uspenski\u\i\ proved that a topological space $X$ is continuously homogeneous if and only if it is rectifiable. A topological space $X$ is called {\em rectifiable} if there is a point $e\in X$ and a homeomorphism $H:X^2\to X^2$
such that $H(x,e)=(x,x)$ and $H(\{x\}\times X)=\{x\}\times X$ for every $x\in X$.

In \cite{BR} Banakh and Repov\v s observed that a topological space $X$ is rectifiable (and hence continuous homogeneous) if and only if $X$ is homeomorphic to a topological lop. A topological lop is a special kind of a topological magma. A ({\em topological}) {\em magma} is a (topological) space $X$ endowed with a (continuous) binary operation $X\times X\to X$, $(x,y)\mapsto xy$. A point $e\in X$ of a magma is called {\em a unit} of $X$ if $xe=x=ex$ for any $x\in X$. It is standard to show that each magma can have at most one unit.

A topological magma $X$ is called a {\em topological lop} if it has a unit and the map $h:X\times X\to X\times X$, $h:(x,y)\to(x,xy)$, is a homeomorphism. This definition implies that for each point $a$ of a topological lop $X$ the left shift $h_a:X\to X$, $h_a:x\mapsto ax$, is a homeomorphism with $h_a(e)=a$. So, topological lops are topologically (and continuously) homogeneous. Topological lops are ``left'' generalizations of topological loops. By definition, a {\em topological loop} is a topological magma $X$ with unit such that the maps $X^2\to X^2$, $(x,y)\mapsto(x,xy)$, and  $X^2\to X^2$, $(x,y)\mapsto(x,yx)$, are homeomorphisms. This definition implies that in a topological loop, left and right shifts are homeomorphisms. More information on topological loops and topological lops can be found in \cite{HS}, \cite{HM},  \cite{BR}, and \cite{Ban}. It is clear that each topological group is a topological loop and each topological loop is a topological lop. Moreover, a topological lop is a topological group if and only if it is associative (see \cite[1.1.2]{Rob}).

So, we get the implications
$$
\xymatrix{
\mbox{associative topological lop }\ar@{<=>}[r]
&\mbox{ topological group }\ar@{=>}[r]
&\mbox{ topological loop }\ar@{=>}[r]
&\mbox{ topological lop.}
}$$

The following equivalence was proved in \cite{BR}.

\begin{proposition}\label{p:rectilop} For a topological space $X$ the following conditions are equivalent:
\begin{enumerate}
\item $X$ is rectifiable;
\item $X$ is continuously homogeneous;
\item $X$ is homeomorphic to a topological lop.
\end{enumerate}
\end{proposition}

Rectifiable topological spaces share many common properties with topological groups. In particular, a rectifiable space is regular (and metrizable) if and only if it is satisfies the separation axiom $T_0$ (and is first countable), see \cite{Gul}.

Let $X,Y$ be topological spaces. Any continuous binary operation $p:Y\times Y\to Y$ induces a binary operation $P:C(X,Y)\times C(X,Y)\to C(X,Y)$ of the set $C(X,Y)$ of all continuous functions from $X$ to $Y$. The operation $P$ assigns to each pair $(f,g)\in C(X,Y)^2$ the function $P(f,g):X\to Y$, $P(f,g):x\mapsto p(f(x),g(x))$. The continuity of the function $P(f,g)$ follows from the observation that $P(f,g)=p\circ (f,g)$ where $(f,g):X\to Y\times Y$, $(f,g):x\mapsto (f(x),g(x))$. The following proposition shows that the induced operation $P$ is continuous with respect to any $\I$-open topology on $C(X,Y)$.

\begin{proposition}\label{p:magmaC} Let $X$ be a topological space and $\I$ be an ideal of compact Hausdorff subspaces of $X$. For any continuous binary operation $p:Y\times Y\to Y$ on a topological space $Y$ the induced operation $P:C_\I(X,Y)\times C_\I(X,Y)\to C_\I(X,Y)$ is continuous.
\end{proposition}

\begin{proof} Fix any pair of functions $(f,g)\in C_\I(X,Y)^2$ and let $O_{fg}$ be a neighborhood of the function $fg=P(f,g)$ in $C_\I(X,Y)$. We lose no generality assuming that $O_{fg}$ is of subbasic form $O_{fg}=[K;U]$ for some compact set $K\in\I$ and some open set $U\subset Y$. Consider the map $(f,g):X\to Y\times Y$, $(f,g):x\mapsto (f(x),g(x))$, and observe that $fg=p\circ (f,g)$. The inclusion $fg(K)\subset U$ implies that the compact set $\{(f(x),g(x)):x\in K\}\subset Y\times Y$ is contained in the open set $W=\{(y,y')\in Y\times Y:p(y,y')\in U\}$. For every point $x\in K$ find open subsets $O_{f(x)}\ni f(x)$ and $O_{g(x)}\ni g(x)$ of $Y$ such that $O_{f(x)}\times O_{g(x)}\subset W$. The continuity of the functions $f,g$ yields an open neighborhood $O_x\subset X$ of $x$ such that $f(O_x)\subset O_{f(x)}$ and $g(O_x)\subset O_{g(x)}$. By the compactness of $K$ the open cover $\{O_x:x\in K\}$ of $K$ contains a finite subcover $\{O_{x_1},\dots,O_{x_n}\}$. Using the  regularity of the compact Hausdorff space $K\in\I$, we can find a cover $\{K_1,\dots,K_n\}$ of $K$ by compact subsets $K_i\subset O_{x_i}$, $1\le i\le n$. Then $O_f=\bigcap_{i=1}^n[K_i,O_{f(x_i)}]$ and $O_g=\bigcap_{i=1}[K_i,O_{g(x_i)}]$ are open neighborhoods of the functions $f$ and $g$ in the function space $C_\I(X,Y)$ such that $P(O_f\times O_g)\subset O_{fg}$. This completes the proof of the continuity of the binary operation $P$ at $(f,g)$.
\end{proof}

Proposition~\ref{p:magmaC} means that for any topological magma $Y$ the function space $C_\I(X,Y)$ has the natural structure of a topological magma.

\begin{proposition}\label{p4.3} Let $X$ be a topological space, $\I$ be an ideal of compact Hausdorff subspaces of $X$ and $Y$ be a topological magma. If $Y$ is a topological lop (topological loop, topological group), then so is the topological magma $C_\I(X,Y)$.
\end{proposition}

\begin{proof} Let $p:Y\times Y\to Y$ stands for the binary operation of the topological magma $Y$.
If $e\in Y$ is a unit of $Y$, then the definition of the induced binary operation $P$ on $C_\I(X,Y)$ guarantees that the constant function $\bar e:X\to\{e\}\subset X$ is the unit of the topological magma $C_\I(X,Y)$.

If $Y$ is a topological lop, then the map $h:Y\times Y\to Y\times Y$, $h:(x,y)\mapsto (x,p(x,y))$ is a homeomorphism. Consequently, for the inverse homeomorphism $h^{-1}:Y\times Y\to Y\times Y$ there is a continuous binary operation $q:Y\times Y\to Y$ such that $h^{-1}(x,y)=(x,q(x,y))$ for all $x,y\in X$. The equalities $h\circ h^{-1}=\id=h^{-1}\circ h$ imply that $p(x,q(x,y))=y=q(x,p(x,y))$ for any $x,y\in Y$. Let $Q$ be the continuous binary operation on the function space $C_\I(X,Y)$ induced by the operation $q$. The equality $p(x,q(x,y))=y=q(x,p(x,y))$ for $x,y\in Y$ implies the equality $P(f,Q(f,g))=g=Q(f,P(f,g))$ holding for any functions $f,g\in C_\I(X,Y)$. Then the map $$H:C_\I(X,Y)^2\to C_\I(X,Y)^2,\;\;H:(f,g)\mapsto (f,P(f,g))$$is a homeomorphism with inverse $H^{-1}(f,g)=(f,Q(f,g))$. This means that the topological magma $C_\I(X,Y)$ is a topological lop with unit $\bar e$.

By analogy we can prove that the topological magma $C_\I(X,Y)$ is a topological loop (group) if so is the topological magma $Y$.
\end{proof}

Combining Propositions~ \ref{p:rectilop} and \ref{p4.3}, we obtain the following corollary, which give us many examples of topologically homogeneous function spaces.

\begin{corollary}\label{c:rectiF} Let $X$ be a topological space and $\I$ be an ideal of compact Hausdorff subspaces of $X$. For any rectifiable space $Y$ the function space $C_\I(X,Y)$ is rectifiable and hence is continuously homogeneous.
\end{corollary}

Combining Corollaries~\ref{c:homoPyt} and \ref{c:rectiF} we obtain the following result generalizing Corollary~\ref{c:groupPyt}.

\begin{corollary}\label{c:contihomF} Let $X$ be a Hausdorff topological space and $\I$ be a  discretely-complete ideal of compact subsets of $X$. For any continuously homogeneous space $Y$ with the strong Pytkeev property, the function space $C_\I(X,Y)$ has the strong Pytkeev property too.
\end{corollary}

Finally we present an example of a topologically homogeneous (first countable) cosmic space $Y$ for which the function space $C_k(2^\w,Y)$ does dot have the strong Pytkeev property.
Here $2^\w$ denotes the Cantor cube $\{0,1\}^\w$. This example shows that Corollary~\ref{c:contihomF} cannot be generalized to topologically homogeneous spaces.

Our counterexample has the structure of a quasitopological group. By a {\em quasitopological group} we understand a group $G$ endowed with a topology $\tau$ in which the inversion $G\to G$, $g\mapsto g^{-1}$, is continuous and the group operation $G\times G\to G$ is separately continuous (which is equivalent to the continuity of left and right shifts). It is clear that quasitopological groups are topologically homogeneous spaces.

We recall that a topological space $X$ is {\em cosmic} if it is regular and has countable network.
It is known \cite[4.8]{Gru} that a regular space is cosmic if and only if it is a continuous image of a separable metrizable space.

\begin{example} There exists a cosmic quasitopological group $Y$ which is first countable \textup{(}and thus has the strong Pytkeev property\textup{)} but the function space $C_k(2^\w,Y)$ has uncountable $\cs^*$-character \textup{(}and hence fails to have the strong Pytkeev property\textup{)}.
\end{example}

\begin{proof} We shall use the quasitopological group constructed in Example 4.2 of \cite{Ban}.
Let $\IQ$ be the additive group of rational numbers. Endow the group $Y=\IR\times\IQ$ with the shift-invariant topology $\tau$ whose neighborhood base at zero $(0,0)$ consists of the sets
$$\maltese_\e=\{(0,0)\}\cup\{(x,y)\in\IR\times\IQ:|xy|<\e(x^2+y^2)<\e^2\}$$
where $\e>0$. It is easy to see that this topology is regular and the family $\{(a,b)\times\{q\}\colon a,b,q\in\IQ,\;a<b\}$ is a countable network for $Y$. Since the topology $\tau$ is first countable and invariant under the inversion, the group $Y=\IR\times\IQ$ endowed with the topology $\tau$ is a first countable cosmic quasitopological group.

Let $X\subset [0,1]$ be the standard Cantor set in the real line. Since the function spaces $C_k(2^\w,Y)$ and $C_k(X,Y)$ are homeomorphic, it suffices to prove that the space $C_k(X,Y)$ fails to have countable $\cs^*$-network at the identity embedding $f:X\to \IR\times\{0\}\subset Y$, $f:x\mapsto (x,0)$. Assume conversely that $C_k(X,Y)$ has a countable $\cs^*$-network $\mathcal N$ at $f$. We lose no generality assuming that $\mathcal N$ is closed under finite unions and hence is a $\cs$-network at $f$ (see \cite{BZd}).

Let $2^{<\w}=\bigcup_{n\in\w}2^n$ be the set of all finite binary sequences. For a binary sequence $s=(s_0,\dots,s_{n-1})\in 2^{<\w}$ and a number $i\in\{0,1\}$ let $|s|=n$ be the length of $s$ and  $s\hat{\,}i=(s_0,\dots,s_{n-1},i)$ be the {\em concatenation} of $s$ and $i$.

Looking at the standard construction of the Cantor set $X\subset\IR$, we can choose a family $(U_s)_{s\in s^{<\w}}$ of closed-and-open subsets of $X$ such that $U_\emptyset=X$ and for every $s\in 2^{<\w}$ the following conditions are satisfied:
\begin{itemize}
\item[(a)] $U_{s\hat{\,}0}\cup U_{s\hat{\,}1}=U_s$ and $U_{s\hat{\,}0}\cap U_{s\hat{\,}1}=\emptyset$;
\item[(b)] $|x-y|\le 3^{-|s|}$ for any points $x,y\in U_s$;
\item[(c)] $|x-y|\ge 3^{-|s|-1}$ for any points $x\in U_{s\hat{\,}0}$ and $y\in U_{s\hat{\,}1}$.
\end{itemize}

For every $m\in\w$ choose any rational point $y_m\in (2^{-m-1},2^{-m})$ and observe that for any $x\in X\subset\IR$ the sequence $\big\{(x,y_m)\big\}_{m\in\w}$ converges to $(x,0)$ in the space $Y$.

For every point $x\in X$ and number $m\in\w$ let $f_{x,m}:X\to Y$ be the continuous function defined by
$$f_{x,m}(z)=\begin{cases}
(z,y_m)&\mbox{if $x\in U_s$ for some $s\in 2^m\subset 2^{<\w}$},\\
(z,0)&\mbox{otherwise}.
\end{cases}
$$

Taking into account that the sequence $\big\{(3^{-m},2^{-m-1})\}_{m\in\w}$ converges to $(0,0)$ in $Y$, we can see that for every $x\in X$ the function sequence $(f_{x,m})_{m\in\w}$ converges to $f$ in the function space $C_k(X,Y)$.

Choose a positive $\e$ such that  $\maltese_\e\subset\{(x,y)\in\IR\times\IQ: |y|\ge 2|x| \mbox{ or $|x|\ge 2|y|$}\}$ and observe that $\e<\frac12$.
For every $x\in X$ consider the neighborhood $O_x(f)=\{g\in C_k(X,Y):g(x)\in f(x)+\maltese_\e\}$ of the function $f$ in the function space $C_k(X,Y)$.

Since the family $\mathcal N$ is a $\cs$-network at $f=\lim_{m\to\infty}f_{x,m}$,
for every $x\in 2^\w$ there are a set $N_x\in\mathcal N$ and a number $m_x\in\w$ such that $\{f_{x,m}\}_{m\ge m_x}\subset N_x\subset O_x(f)$. Since the set $\mathcal N\times \w$ is countable, by the Pigeonhole Principle,
for some $(N,m)\in\mathcal N\times \w$ the subset $X'=\{x\in X: N_x=N,\;m_x=m\}\subset \IR$ is uncountable and hence contains two points $x,x'\in X$ with $|x-x'|<\min\{\e,3^{-m-1}\}$. Find a unique $k\in\w$ such that $2^{-k-1}<|x-x'|\le 2^{-k}$ and observe that $k\ge m$ (as $2^{-k-1}<|x-x'|<3^{-m-1}$). Then $f_{x',m},f_{x,m}\in P_{x'}=P_x\subset O_{x}(f)$. The condition (c) and the inequality $|x-x'|<3^{-m-1}$ imply  that $x',x\in U_s$ for some $s\in 2^m$. Then $(x',y_m)=f_{x',m}(x)\in f(x)+\maltese_\e=(x,0)+\maltese_\w$ yields $(x'-x,y_m)\in\maltese_\e$, which is not possible as
$$\frac12=\frac{2^{-m-1}}{2^{-m}}<\frac{|y|}{|x'-x|}<\frac{2^{-m}}{2^{-m-1}}=2$$contradicting the choice of $\e>0$.
\end{proof}

\section{The structure and metrizability of topological lops with the strong Pytkeev property}\label{s5}

In this section we establish some structure properties of topological groups (and lops) with the strong Pytkeev property. In fact all the arguments work to only for topological groups but for the more general class of topological lops (or rectifiable spaces). From now on all rectifiable spaces (and topological lops) are assumed to be regular (equivalently, $T_0$-spaces).
 Combining Proposition~\ref{first:sak} with a result of A.Gulko (saying that each first countable rectifiable space is metrizable) we get the following metrizability criterion.

\begin{proposition} A rectifiable space is metrizable $X$ if and only if $X$ has both the strong Pytkeev property and countable fan open-tightness.
\end{proposition}

Next, we reveal the structure of sequential topological groups (more generally, sequential topological lops) with the strong Pytkeev property proving that they contain an open $sk_\w$-subgroup (resp. open $sk_\w$-sublop).

A topological space $X$ is called a {\em $k_\w$-space} (resp. {\em $sk_\w$-space}) if there is a countable family $\K$ of compact (resp. compact metrizable) subspaces of $X$ such that a subset $F\subset X$ is closed in $X$ if and only if for every $K\in\K$ the intersection $K\cap F$ is closed in $K$. A topological group will be called a {\em $k_\w$-group} (resp. a {\em $sk_\w$-group}) is its underlying topological space is a $k_\w$-space (resp. $sk_\w$-space).
Observe that a $k_\w$-space $X$ is an $sk_\w$-space if and only if it is {\em submetrizable} (which means that $X$ admits a continuous metric). It is easy to see that each $sk_\w$-space, being a sequential $\aleph_0$-space, is a $\Pyt_0$-space and hence has the strong Pytkeev property.

In \cite{BZd} Banakh and Zdomskyy proved that each sequential topological groups with countable $\cs^*$-character is metrizable or contains an open $sk_\w$-subgroup. In \cite{BR} this result was generalized to topological lops and rectifiable spaces.

\begin{theorem}\label{t:normrec} For a sequential rectifiable space $X$ the following conditions are equivalent:
\begin{enumerate}
\item $X$ has the strong Pytkeev property;
\item $X$ has countable $\cs^*$-character;
\item $X$ is metrizable or contains a clopen rectifiable $sk_\w$-subspace.
\end{enumerate}
\end{theorem}

\begin{proof} The implication $(3)\Ra(1)\Ra(2)$ are trivial while $(2)\Ra(3)$ is not trivial but is proved in \cite{BR}.
\end{proof}

\begin{corollary}\label{t:normrec} For a sequential Lindel\"of rectifiable space $X$ the following conditions are equivalent:
\begin{enumerate}
\item $X$ has the strong Pytkeev property;
\item $X$ has countable $\cs^*$-character;
\item $X$  either is metrizable or is an $sk_\w$-space.
\end{enumerate}
\end{corollary}

Our next aim is to show that for locally narrow topological lops the strong Pytkeev property is equivalent to metrizability.

\begin{definition} A subset $B$ of a topological lop $X$ is called
\begin{itemize}
\item {\em narrow} if for any neighborhood $U\subset X$ of the unit $e$, any infinite subset $A\subset B$ contains a point $a\in A$ such that $A\cap aU$ is infinite;
\item {\em bounded} if for any neighborhood $U\subset X$ of the unit $e$ there is a finite set $F\subset X$ such that $B\subset FU$.
\end{itemize}
\end{definition}

It is easy to see that each narrow subset of a topological lop is bounded. The converse is true in topological groups.

\begin{definition} A topological lop $X$ is called
\begin{itemize}
\item {\em locally narrow} if $X$ contains a narrow neighborhood of the unit;
\item {\em locally bounded} if $X$ contains a bounded neighbrhood of the unit;
\item ({\em locally}) {\em precompact} if $X$ is topologically isomorphic to a sublop of a (locally) compact topological lop;
\vskip2pt

\item {\em pseudocompact} if the topological space of $X$ is Tychonoff and each continuous real-valued function on $X$ is bounded;
\item {\em countably-compact} if each infinite subset $A\subset X$ has an accumulation point in $X$;
\item {\em locally countably-compact} if each point $x\in X$ has a countably-compact neighborhood in $X$.
\end{itemize}
\end{definition}
Observe that the three latter properties are topological and do not depend on the algebraic structure of the topological lop.

For any topological lop these properties relate as follows:
$$
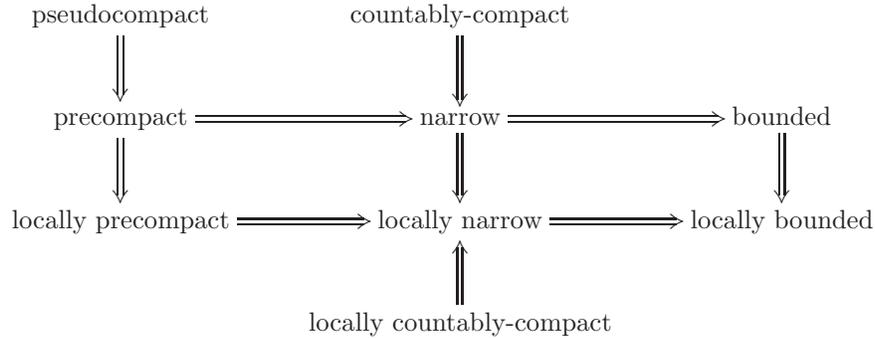
\xymatrix{
\mbox{pseudocompact}\ar@{=>}[d]&\mbox{countably-compact}\ar@{=>}[d]\\
\mbox{precompact}\ar@{=>}[d]\ar@{=>}[r]&\mbox{narrow}\ar@{=>}[d]\ar@{=>}[r]&\mbox{bounded}\ar@{=>}[d]\\
\mbox{locally precompact}\ar@{=>}[r]&\mbox{locally narrow}\ar@{=>}[r]&\mbox{locally bounded}\\
&\mbox{locally countably-compact}\ar@{=>}[u].
}
$$
Non-trivial implications from this diagram are proved in \cite[\S4]{Ban}. It is easy to see that a sublop of a locally narrow topological lop is locally narrow. We do not know if the similar property holds for locally bounded topological lops. But for topological groups we have the following equivalence (see \cite[3.7.I]{ArT}).

\begin{proposition} For any topological group $G$ the following conditions are equivalent:
\begin{enumerate}
\item $G$ is (locally) precompact;
\item $G$ is (locally) narrow;
\item $G$ is (locally) bounded.
\end{enumerate}
\end{proposition}

Now we establish the promised metrization criterion for locally narrow topological lops.

\begin{theorem}\label{t:lnlop} A locally narrow topological lop $X$ has the strong Pytkeev property if and only if $X$ is metrizable.
\end{theorem}

\begin{proof} The ``only if'' part is trivial. To prove the ``if'' part, assume that $X$ is a locally narrow topological lop with the strong Pytkeev property. Then there exists a neighborhood $W$ of the unit $e$ in $X$ such any for any neighborhood $U\subset X$ of $e$, each infinite set $A\subset W$ contains a point $a\in A$ such that $A\cap aU$ is infinite.

Fix any countable Pytkeev network $\mathcal N$ at the unit $e$ of $X$.
We lose no generality assuming that each set $N\in\mathcal N$ is closed in $X$ and contains the point $e$.
Consider the subfamily $\mathcal N'=\{N\in \mathcal N: N$ is nowhere dense in $G\}$ and let $\mathcal N'=\{N_k'\}_{k\in\w}$ be an enumeration of $\mathcal N'$. Use the nowhere density of the sets $N_k'$ to construct a sequence $A=\{a_k\}_{k\in\w}\subset W$ such that $a_n\in W\setminus\bigcup_{k,m<n}a_kN_m'$ for every $n\in\w$. We claim that the set $B=\{a_k^{-1}a_n:k<n\}\subset X$ accumulates at $e$. Indeed, given any neighborhood $U\subset X$ of $e$, we can find a point $a_k\in A$ such that the set  $A\cap a_kU$ is infinite. Then for every point $a_m\in A\cap a_kU$ with $m>k$ we get $a_k^{-1}a_m\in U\cap B$, which means that $B$ accumulates at $e$.
Observe that for any $m\in\w$ the set $B\cap N'_m\subset \{a_k^{-1}a_n:k<n\le m\}$ is finite.
This implies that the family $\mathcal N'=\{N_m'\}_{m\in\w}$ is not a Pytkeev network at $e$.
Taking into account that $\mathcal N$ is a Pytkeev network at $e$, we conclude that each neighborhood $U\subset X$ of $e$ contains a set $N\in\mathcal N\setminus \mathcal N'$. Observe that each (closed) set $N\in\mathcal N\setminus\mathcal N'$ has non-empty interior and hence $N^{-1}N$ is a neighborhood of $e$. Then $\mathcal B_e=\{N^{-1}N:N\in\mathcal N\setminus\mathcal N'\}$ is a countable neighborhood base at $e$. So, $X$ is first countable at $e$. By \cite{Gul} the topological lop $X$, being rectifiable and first countable, is metrizable.
\end{proof}

\begin{corollary}\label{c:lnlop} A locally narrow topological lop $X$ is metrizable if and only if $X$ contains a dense sublop with the strong Pytkeev property.
\end{corollary}

\begin{proof} The ``only if'' part is trivial. To prove the ``if'', assume that a narrow topological lop $X$ contains a dense sublop $H$ with the strong Pytkeev property. By Theorem~\ref{t:lnlop}, the topological lop $X$ is first countable. The regularity of $X$ and the density of $H$ in $X$ implies that $X$ is first countable at the unit $e$. By \cite{Gul}, the topological lop $X$, being a first countable rectifiable space, is metrizable.
\end{proof}

For topological groups Corollary~\ref{c:lnlop} implies:

\begin{corollary}\label{c:lngrp} A locally precompact topological groups $X$ is metrizable if and only if $X$ contains a dense subgroup with the strong Pytkeev property.
\end{corollary}

\section{Compact sets in topological groups and rectifiable spaces with the strong Pytkeev property}\label{s6}

In this section we consider some properties of countably compact subsets in topological groups (more generally, topological lops) with the strong Pytkeev property.

The following important fact was proved in \cite{BR}.

\begin{theorem} If a sequential rectifiable space $X$ has countable $\cs^*$-character, then each  countably compact subspace of $X$ is metrizable.
\end{theorem}

Without sequentiality we can prove a bit weaker result.

\begin{theorem}\label{t:rectcs} Let $X$ be a rectifiable space with countable $\cs^*$-character.
\begin{enumerate}
\item A compact subspace $K$ of $X$ is metrizable if and only if $K$ is sequential.
\item A countably compact subspace $K$ of $X$ is metrizable if and only if its square $K\times K$ is Fr\'echet-Urysohn and has countable fan open-tightness at each point of the diagonal $\Delta_K$.
\end{enumerate}
\end{theorem}

\begin{proof} The first statement of this theorem was proved in \cite{BR}. The second statement trivially follows Lemma~\ref{comp1}(2) (below) and the fact that for a topological lop $X$ and the division operation $q:X\times X\to X$, $q:(x,y)\mapsto x^{-1}y$, the preimage $q^{-1}(e)$ coincides with the diagonal $\Delta_X$ of $X\times X$.
\end{proof}

\begin{lemma}\label{comp1} Let $X$ be a countably compact regular space such that the square $X\times X$ has countable fan open-tightness at each point of the diagonal $\Delta_X=\{(x,y)\in X\times X:x=y\}$, and $q:X\times X\to Y$ be a continuous map into a topological $T_1$-space $Y$ such that $\Delta_X=q^{-1}(e)$ for some point $e\in Y$. The space $X$ is metrizable if one of the following conditions is satisfied:
\begin{enumerate}
\item $Y$ has the strong Pytkeev property at $e$;
\item $X\times X$ is Fr\'echet-Urysohn at each point of $\Delta_X$ and $Y$ has countable $\cs^*$-character at $e$.
\end{enumerate}
\end{lemma}

\begin{proof} Assuming that the space $X$ is not metrizable we shall prove that the conditions (1) and (2) are not satisfied. By Chaber's Theorem \cite[2.14]{Gru}, each countably compact regular space with  $G_\delta$-diagonal is metrizable. So, $X$ does not have $G_\delta$-diagonal.

Let $\mathcal N$ be a countable family of closed subsets, closed under finite unions. We shall assume that $\mathcal N$ is a $\cs^*$-network at $e$ if $Y$ has countable $\cs^*$-character at $e$, and  $\mathcal N$ is a Pytkeev network at $e$ if $Y$ has the strong Pytkeev property at $e$.

For every $N\in\mathcal N$ consider the interior $U_N$ of the set $q^{-1}(N)$ in $X\times X$ and let $\mathcal N'=\{N\in\mathcal N:\Delta_X\subset U_N\}$. Since the diagonal $\Delta_X$ is not a $G_\delta$-set in $X\times X$, the intersection $\bigcap_{N\in\mathcal N'}U_N$ contains some pair $(u,v)\notin\Delta_X$.

Now consider the point $y=q(u,v)\ne e$ and let $\mathcal N_y=\{N\in\mathcal N:y\notin N\}$. Since the family $\mathcal N_y$ is countable and closed under finite unions, there is an increasing sequence of sets $\{N_k\}_{k\in\w}\subset\mathcal N_y$ such that each set $N\in\mathcal N_y$ is contained in some set $N_k$.

The choice of the point $(u,v)\in\bigcap_{N\in\mathcal N'}U_N\subset\bigcap_{N\in\mathcal N'}q^{-1}(N)$ guarantees that for every $k\in\w$ the set $N_k$ does not belong to the subfamily $\mathcal N'$ and hence $\Delta_X\not\subset U_{N_k}$, which implies that the closure $\bar A_k$ of the open set $A_k=X^2\setminus q^{-1}(N_k)=q^{-1}(Y\setminus N_k)$ meets the diagonal $\Delta_X$.
Then $(\Delta_X\cap\bar A_k)_{k\in\w}$ is a decreasing sequence of
non-empty closed subset of the diagonal $\Delta_X$. By the countable compactness of $\Delta_X$, the intersection $\bigcap_{k\in\w}\Delta_X\cap\bar A_k$ contains some point $(a,a)$.

By the countable fan open-tightness of $X\times X$ at the point $(a,a)$, in each open set $A_k$ we can choose a finite subset $F_k\subset A_k\subset q^{-1}(Y\setminus N_k)$ so that the union $F=\bigcup_{k\in\w}F_k$ contains the point $(a,a)$ in its closure. By the continuity of $q$, the set $q(F)=\bigcup_{k\in\w}q(F_k)\subset \bigcup_{k\in\w}(Y\setminus N_k)$ contains the point $e$ in its closure. Now we consider two cases.
\smallskip

1. The space $Y$ has the strong Pytkeev property at $e$. By our assumption, the family $\mathcal N$ is a Pytkeev network at $e$. So, we can find a set $N\in\mathcal N_y$ such that $N\cap q(F)$ is infinite. The choice of the sequence $(N_k)_{k\in\w}$ guarantees that $N\subset N_m$ for some $m$ and then $$N_m\cap q(F)=\bigcup_{k\in\w}N_m\cap q(F_k)=\bigcup_{k\in\w}N_m\cap(q(F_k)\setminus N_k)\subset \bigcup_{k<m}q(F_k)$$ is finite, which is a contradiction.
\smallskip

2. The space $X\times X$ is Fr\'echet-Urysohn at each point of $\Delta_X$ and $Y$ has countable $\cs^*$-character at $e$. Since the space $X\times X$ is Fr\'echet-Urysohn at the point $(a,a)$,
we can find an infinite subset $E\subset F$ which converges to $(a,a)$ in the sense that each neighborhood $O_{(a,a)}\subset X\times X$ of $(a,a)$ contains all but finitely many point of the infinite set $E$. By the continuity of $q$ the image $q(E)$ is a sequence convergent to the point $e$.
By our assumption, the family $\mathcal N$ is a $\cs^*$-network at $e$. So, it contains a set $N\in\mathcal N$ such that $N\subset Y\setminus\{y\}$ and $N\cap q(E)$ is infinite. By the choice of the sequence $(N_k)_{k\in\w}$, the set $N\in\mathcal N_y$  is contained in some set $N_m$ and then the intersection
$$N\cap q(E)\subset N_m\cap q(F)\subset \bigcup_{k< m}q(F_k)$$ is finite, which is a desirable contradiction.
\end{proof}

Lemma~\ref{comp1}(1) implies the following metrizability criterion.

\begin{theorem}\label{t:comp1} Let $X$ be rectifiable space with the strong Pytkeev property. A countably compact subspace $K$ of $X$ is metrizable if and only if $K\times K$ has countable fan open-tightness at each point of the diagonal $\Delta_K$.
\end{theorem}

As an application of Theorem~\ref{t:rectcs}(2) we get:

\begin{corollary} The countably compact space $[0,\w_1)$ is first countable but does not embed into a rectifiable space with countable $\cs^*$-character.
\end{corollary}

\end{document}